\documentclass[11pt,reqno]{amsart}

\usepackage[T1]{fontenc}
\usepackage{lmodern}
\usepackage{amsmath,amssymb,amsthm}
\usepackage{mathtools}
\usepackage[expansion=false]{microtype}
\usepackage{booktabs}
\usepackage[margin=1.12in]{geometry}
\usepackage{xurl}
\usepackage[hidelinks]{hyperref}
\usepackage{etoolbox}

\makeatletter
\patchcmd{\@setauthors}{\MakeUppercase{\authors}}{\authors}{}%
  {\PackageError{authorcase}{Unable to preserve author capitalization}{}}
\makeatother

\DeclareFontFamily{U}{wncy}{}
\DeclareFontShape{U}{wncy}{m}{n}{<-> wncyr10}{}
\DeclareSymbolFont{cyrletters}{U}{wncy}{m}{n}
\DeclareMathSymbol{\Sha}{\mathord}{cyrletters}{"58}

\newcommand{\Q}{\mathbb Q}
\newcommand{\Z}{\mathbb Z}
\newcommand{\R}{\mathbb R}
\newcommand{\F}{\mathbb F}
\newcommand{\Sel}{\operatorname{Sel}}
\newcommand{\rk}{\operatorname{rank}}
\newcommand{\corank}{\operatorname{corank}}
\newcommand{\Reg}{\operatorname{Reg}}
\newcommand{\tors}{\mathrm{tors}}
\newcommand{\an}{\mathrm{an}}

\newcommand{\hphi}{\widehat\phi}
\newcommand{\eps}{\varepsilon}
\DeclareMathOperator{\ord}{ord}

\newtheorem{theorem}{Theorem}[section]
\newtheorem{proposition}[theorem]{Proposition}
\newtheorem{lemma}[theorem]{Lemma}
\newtheorem{corollary}[theorem]{Corollary}
\theoremstyle{definition}
\newtheorem{hypothesis}[theorem]{Hypothesis}

\theoremstyle{remark}
\newtheorem{remark}[theorem]{Remark}
\numberwithin{equation}{section}

\title[Algebraic and analytic ranks of twin-prime elliptic curves]{On the algebraic and analytic ranks of the\newline twin-prime elliptic curve $y^2=x(x-2)(x-p)$}
\author{Ruihan Chen}
\date{October 7, 2026}
\subjclass[2020]{Primary 11G05; Secondary 11G40}
\keywords{Elliptic curves, twin primes, Selmer groups, Birch--Swinnerton-Dyer conjecture, Cassels--Tate pairing}
\hypersetup{
 pdftitle={On the algebraic and analytic ranks of the twin-prime elliptic curve y^2=x(x-2)(x-p)},
 pdfauthor={Ruihan Chen},
 pdfsubject={A conditional application of low-corank BSD to a classical Selmer calculation},
 pdfkeywords={elliptic curves, Selmer groups, BSD, twin primes}
}

\begin{document}

\begin{abstract}
Let $p\ge 7$ and suppose that $p$ and $p-2$ are prime. We study
$E_p:y^2=x(x-2)(x-p)$ using the classical $2$-Selmer calculation of
Qiu--Zhang and the Cassels--Tate pairing. These give $2^\infty$-Selmer
corank one for $p\equiv3,5\pmod8$ and corank zero for
$p\equiv7\pmod8$. Assuming the low-corank Birch--Swinnerton-Dyer
statement announced in the October 2026 OpenAI mathematics release, we
deduce equality of the analytic and algebraic ranks in these cases,
finiteness of the full Tate--Shafarevich group, and the exact BSD
leading-term formula. The $2$-primary Tate--Shafarevich group is trivial,
and the remaining factor has odd square order. We identify precisely
the obstruction left in the class $p\equiv1\pmod8$. An appendix gives a
local descent proof of the known Selmer dimensions in the coordinates
used here.
\end{abstract}

\maketitle
\markboth{Ruihan Chen}{\shorttitle}
\enlargethispage{2pt}

\section{Introduction and the external input}

Throughout the paper, $p\ge7$ is a prime such that
$q:=p-2$ is also prime, and
\begin{equation}\label{eq:curve}
 E_p:\quad y^2=x(x-2)(x-p).
\end{equation}
We write $r(E)=\rk E(\Q)$, $r_{\an}(E)=\ord_{s=1}L(E,s)$, and
\[
 \rho_2(E)=\corank_{\Z_2}\Sel_{2^\infty}(E/\Q).
\]
Here $\Sel_{2^\infty}$ denotes the full $2$-power Selmer group with the
usual local Kummer conditions. In particular, $\rho_2(E)$ is not the
dimension of the finite group $\Sel_2(E/\Q)$.

The change of variables $x=X+p$ identifies \eqref{eq:curve} with
\begin{equation}\label{eq:translation}
 Y^2=X(X+q)(X+p).
\end{equation}
This is the positive-sign family studied by Qiu--Zhang \cite{QZ}, with
$q$ the smaller of their two prime parameters. Their results determine
the full $2$-Selmer group and the local arithmetic invariants of this
family. Hatley \cite{Hatley} subsequently studied
\eqref{eq:curve} by $2$-descent. Joshi \cite{Joshi} computed its root
number and formulated the following corrected version of a conjecture
of Beers:
\begin{equation}\label{eq:conjecture}
 r(E_p)=
 \begin{cases}
 0,&p\equiv7\pmod8,\\
 1,&p\equiv3,5\pmod8,\\
 0\text{ or }2,&p\equiv1\pmod8.
 \end{cases}
\end{equation}
The correction is necessary because the assertion that the last rank
is always two has the counterexample $p=73$
\cite[Proposition 5.2.1 and Conjecture 5.3.1]{Joshi}.

The purpose of this note is to spell out a pointwise application of a
recently announced low-corank BSD statement to this classical family.
The finite Selmer calculations are due to Qiu--Zhang; the arithmetic
conclusions below combine those calculations with the announced
statement. We isolate the latter as an explicit hypothesis, so that
the precise dependence of the note does not rest on an implicit
acceptance of a newly released proof.

For an elliptic curve $E/\Q$, let $\omega_E$ be a global minimal
N\'eron differential and put
\[
 \Omega_E=\int_{E(\R)}|\omega_E|.
\]
This is the period over \emph{all} real components. We use the
N\'eron--Tate regulator $\Reg_E$, with $\Reg_E=1$ in rank zero, and
write $c_\ell(E)$ for the Tamagawa number at a finite prime $\ell$.

\begin{hypothesis}[Low-corank BSD at $2$]\label{hyp:bsd}
Let $E/\Q$ be an elliptic curve with $\rho_2(E)\in\{0,1\}$. Then
\[
 r(E)=r_{\an}(E)=\rho_2(E),\qquad \#\Sha(E/\Q)<\infty,
\]
and, writing $r=\rho_2(E)$, one has the exact equality
\begin{equation}\label{eq:generalbsd}
 \frac{L^{(r)}(E,1)}{r!}
 =\frac{\Omega_E\Reg_E\,\#\Sha(E/\Q)
             \prod_{\ell}c_\ell(E)}{\#E(\Q)_{\tors}^{\,2}}.
\end{equation}
\end{hypothesis}

Hypothesis~\ref{hyp:bsd} is the specialization at $2$ of the statement
announced in \cite{OAI}.
All references to this input concern the release available on
October~7, 2026. The unconditional results in Sections~\ref{sec:selmer}
and~\ref{sec:remaining}, and in the appendix, do not use
Hypothesis~\ref{hyp:bsd}.

\begin{theorem}\label{thm:main}
Assume Hypothesis~\ref{hyp:bsd}. Let $p,q$ be as above and suppose
$p\equiv3,5,$ or $7\pmod8$. Then
\begin{equation}\label{eq:mainrank}
 r(E_p)=r_{\an}(E_p)=\rho_2(E_p)=
 \begin{cases}
 1,&p\equiv3,5\pmod8,\\
 0,&p\equiv7\pmod8.
 \end{cases}
\end{equation}
Moreover,
\begin{equation}\label{eq:mainsha}
 \#\Sha(E_p/\Q)<\infty,\qquad
 \Sha(E_p/\Q)[2^\infty]=0,
\end{equation}
and $E_p(\Q)\simeq\Z^r\oplus(\Z/2\Z)^2$, where $r$ is given by
\eqref{eq:mainrank}. If $\Omega_p=\int_{E_p(\R)}|dx/(2y)|$, then
\begin{equation}\label{eq:explicitbsd}
 \frac{L^{(r)}(E_p,1)}{r!}
       =\frac{\Omega_p\Reg_{E_p}}{2}\,\#\Sha(E_p/\Q).
\end{equation}
The order of $\Sha(E_p/\Q)$ is an odd perfect square.
\end{theorem}

Thus, relative to the stated input, the rank-one part of
\eqref{eq:conjecture} follows together with the full leading-term
formula. The rank-zero assertion for $p\equiv7\pmod8$ is already
unconditional; the input supplies the analytic and full BSD conclusions.
No assertion that \eqref{eq:conjecture} is completely proved is made:
Section~\ref{sec:remaining} describes the remaining possibility when
$p\equiv1\pmod8$.

\section{Finite Selmer groups and infinite Selmer coranks}\label{sec:selmer}

\subsection{The classical descent calculation}

For an elliptic curve over $\Q$, define its reduced $2$-Selmer
dimension by
\[
 s_2(E)=\dim_{\F_2}\Sel_2(E/\Q)-\dim_{\F_2}E(\Q)[2].
\]
The word ``reduced'' here refers only to the subtraction of the rational
$2$-torsion dimension. It does not mean that any local Selmer condition
has been removed.

\begin{proposition}[Qiu--Zhang]\label{prop:classical}
For the curve $E_p$ one has
\begin{equation}\label{eq:selmerdim}
 \dim_{\F_2}\Sel_2(E_p/\Q)=
 \begin{cases}
 4,&p\equiv1\pmod8,\\
 3,&p\equiv3,5\pmod8,\\
 2,&p\equiv7\pmod8.
 \end{cases}
\end{equation}
In particular, $s_2(E_p)$ is respectively $2,1,0$ in these three cases.
\end{proposition}

\begin{proof}
The rational $2$-torsion consists of
$O,(0,0),(2,0),(p,0)$ and has dimension two. Under
\eqref{eq:translation}, Theorem~3(2) of the preprint version of
\cite{QZ} gives the Selmer dimensions, indexed by the smaller prime
$q$. Substituting $q=p-2$ gives \eqref{eq:selmerdim}.
For a direct local calculation with a different rational
$2$-isogeny, see Appendix~\ref{app:descent}.
\end{proof}

\subsection{The divisible part of the Tate--Shafarevich group}

We recall the exact sequences
\begin{align}
 0&\longrightarrow E(\Q)/2E(\Q)
   \longrightarrow\Sel_2(E/\Q)
   \longrightarrow\Sha(E/\Q)[2]\longrightarrow0,
   \label{eq:kummerfinite}\\
 0&\longrightarrow E(\Q)\otimes\Q_2/\Z_2
   \longrightarrow\Sel_{2^\infty}(E/\Q)
   \longrightarrow\Sha(E/\Q)[2^\infty]\longrightarrow0.
   \label{eq:kummerinfinite}
\end{align}
These are the ordinary Kummer sequences with the local conditions
imposed; see \cite[Chapter X]{Silverman}.

Set $G=\Sha(E/\Q)[2^\infty]$. This is a cofinitely generated torsion
$\Z_2$-module, so its maximal divisible subgroup is
$D\simeq(\Q_2/\Z_2)^\delta$ for an integer $\delta\ge0$, and
$F=G/D$ is finite. Since divisible abelian groups are injective,
$G\simeq D\oplus F$ as abelian groups. The splitting is not canonical.

\begin{lemma}\label{lem:parity}
For every elliptic curve $E/\Q$, there is an integer $m\ge0$ such that
\begin{equation}\label{eq:parityidentity}
 s_2(E)=\rho_2(E)+2m.
\end{equation}
More precisely,
\begin{equation}\label{eq:finitequotient}
 F\simeq\bigoplus_{i=1}^{m}(\Z/2^{a_i}\Z)^2
 \quad\text{with }a_i\ge1,
 \qquad
 \rho_2(E)=r(E)+\delta.
\end{equation}
No finiteness assumption on $\Sha(E/\Q)$ is required.
\end{lemma}

\begin{proof}
The Cassels--Tate pairing is alternating for elliptic curves, and its
kernel is the maximal divisible subgroup. Its restriction to the
finite $2$-primary quotient therefore gives a nondegenerate alternating
pairing $F\times F\to\Q_2/\Z_2$; see \cite[Section 1]{PS}.

For completeness, a finite abelian $2$-group with such a pairing is an
orthogonal direct sum of pairs of equal cyclic groups. Indeed, choose
$x\in F$ of maximal order $2^a$. Nondegeneracy provides $y\in F$ for
which $\langle x,y\rangle$ has order $2^a$. The subgroup generated by
$x,y$ is isomorphic to $(\Z/2^a\Z)^2$, its restricted pairing is
nondegenerate, and it splits off orthogonally. Induction proves the
first assertion in \eqref{eq:finitequotient}. In particular,
$\dim_{\F_2}F[2]=2m$.

Taking dimensions in \eqref{eq:kummerfinite} and coranks in
\eqref{eq:kummerinfinite} gives
\[
 s_2(E)=r(E)+\dim_{\F_2}G[2]=r(E)+\delta+2m,
 \qquad \rho_2(E)=r(E)+\delta.
\]
Subtracting yields \eqref{eq:parityidentity}.
\end{proof}

\begin{corollary}\label{cor:coranks}
Unconditionally,
\begin{equation}\label{eq:coranktable}
 \rho_2(E_p)=
 \begin{cases}
 0\text{ or }2,&p\equiv1\pmod8,\\
 1,&p\equiv3,5\pmod8,\\
 0,&p\equiv7\pmod8.
 \end{cases}
\end{equation}
If $p\equiv3,5\pmod8$, the only possibilities for
$(r(E_p),\Sha(E_p/\Q)[2^\infty])$ are
\begin{equation}\label{eq:odddichotomy}
 (1,0)\qquad\text{or}\qquad(0,\Q_2/\Z_2).
\end{equation}
If $p\equiv7\pmod8$, then $r(E_p)=0$ and
$\Sha(E_p/\Q)[2^\infty]=0$.
\end{corollary}

\begin{proof}
Substitute the values $s_2(E_p)=2,1,0$ into
\eqref{eq:parityidentity}. When $s_2=1$, one must have
$\rho_2=1$ and $m=0$, so $F=0$ and $r+\delta=1$.
When $s_2=0$, all of $\rho_2,m,r,\delta$ vanish. The remaining
case gives $\rho_2\in\{0,2\}$.
\end{proof}

\begin{remark}
This proves the relevant Selmer-corank statement without using a
root-number formula or a $2$-parity theorem. In particular, it avoids
replacing the parity of a Selmer corank by the parity of the
Mordell--Weil rank. The distinction is exactly the second possibility
in \eqref{eq:odddichotomy}.
\end{remark}

\section{The pointwise BSD consequences}\label{sec:bsd}

\subsection{Local factors and period conventions}

\begin{proposition}\label{prop:local}
Equation~\eqref{eq:curve} is globally minimal, with
\begin{equation}\label{eq:invariants}
 \Delta(E_p)=64p^2q^2,\quad
 c_4(E_p)=16(p^2-2p+4),\quad N(E_p)=2^5pq.
\end{equation}
The rational torsion subgroup is $(\Z/2\Z)^2$, and
\begin{equation}\label{eq:tamagawa}
 c_2(E_p)=c_p(E_p)=c_q(E_p)=2,
 \qquad c_\ell(E_p)=1\quad(\ell\notin\{2,p,q\}).
\end{equation}
The differential $\omega_p=dx/(2y)$ is a global minimal N\'eron
differential. The real group $E_p(\R)$ has two connected components.
\end{proposition}

\begin{proof}
The discriminant and $c_4$ follow directly from the Weierstrass
coefficients $a_2=-(p+2)$ and $a_4=2p$. At each prime the
valuation of the discriminant is less than $12$, so the integral
model is minimal. Notice that the discriminant in
\eqref{eq:invariants} is \emph{positive}.

At $p$ and $q$, $c_4$ is a unit and the minimal discriminant has
valuation two. The reduction is therefore of type $I_2$, and the
Tamagawa number is two in both the split and nonsplit cases. At $2$,
Tate's algorithm gives type $\mathrm{III}$, with Tamagawa number two
and conductor exponent five. These local computations are also recorded
in \cite[Theorem 2 and Corollary 1]{QZ}, after
\eqref{eq:translation}. They prove the conductor and
\eqref{eq:tamagawa}.

Since $p,q>3$ are twin primes, $p\equiv1\pmod3$ and
$q\equiv2\pmod3$. Reduction at the good prime $3$ has precisely the
four points $O,(0,0),(1,0),(2,0)$. The kernel of reduction on
$E_p(\Q_3)$ is a torsion-free formal group, so rational torsion
injects into this group of four points. The four rational
$2$-torsion points already account for it. Finally, the three roots
$0,2,p$ are distinct and real, which gives two real components.
\end{proof}

\subsection{Proof of the main theorem}

\begin{proof}[Proof of Theorem~\ref{thm:main}]
Corollary~\ref{cor:coranks} verifies the hypothesis
$\rho_2(E_p)\in\{0,1\}$, with its exact value in each of the three
residue classes. Hypothesis~\ref{hyp:bsd} therefore gives
\eqref{eq:mainrank}, finiteness of the full Tate--Shafarevich group,
and \eqref{eq:generalbsd} for $E_p$.

When $p\equiv3,5\pmod8$, the rank is one. The first Kummer sequence
and Proposition~\ref{prop:classical} give
\[
 \dim_{\F_2}\Sha(E_p/\Q)[2]=s_2(E_p)-r(E_p)=1-1=0.
\]
A nonzero $2$-primary torsion group has a nonzero element of order two;
hence $\Sha(E_p/\Q)[2^\infty]=0$. For $p\equiv7\pmod8$,
this vanishing is already in Corollary~\ref{cor:coranks}.

Proposition~\ref{prop:local} gives
$\prod_\ell c_\ell(E_p)=8$ and $\#E_p(\Q)_{\tors}^2=16$.
Substitution into \eqref{eq:generalbsd} yields
\eqref{eq:explicitbsd}. Since the full Tate--Shafarevich group is
finite, its Cassels--Tate pairing is nondegenerate and alternating;
its order is therefore a square. The vanishing of the $2$-primary
part makes this square odd. The description of $E_p(\Q)$ follows
from the Mordell--Weil theorem and the torsion calculation.
\end{proof}

\begin{corollary}\label{cor:period}
Under the assumptions of Theorem~\ref{thm:main}, put
\[
 \Omega_p^+=\int_{E_p(\R)^0}|\omega_p|=\frac{\Omega_p}{2}.
\]
Then
\begin{equation}\label{eq:oddsquare}
 \frac{L^{(r)}(E_p,1)}{r!\,\Omega_p^+\Reg_{E_p}}
       =\#\Sha(E_p/\Q)
\end{equation}
is an odd positive square integer. With the total real period
$\Omega_p$ instead, the same normalized leading coefficient has
$2$-adic valuation $-1$.
\end{corollary}

\begin{proof}
Translation identifies the two real components and preserves the
invariant measure. Thus each has half the total period. The claims
follow from \eqref{eq:explicitbsd}.
\end{proof}

\section{What remains when \texorpdfstring{$p\equiv1\pmod8$}{p is 1 modulo 8}}\label{sec:remaining}

The Selmer dimension in this case gives a particularly concrete
form of the remaining issue.

\begin{proposition}\label{prop:remaining}
Suppose $p\equiv1\pmod8$. Exactly one of the following alternatives
holds, where the displayed isomorphisms are isomorphisms of abelian
groups:
\begin{equation}\label{eq:fourcases}
\begin{array}{ccl}
 \rho_2(E_p)&r(E_p)&\Sha(E_p/\Q)[2^\infty]\\[2pt]
 \hline
 0&0&(\Z/2^a\Z)^2\quad\text{for some }a\ge1,\\
 2&0&(\Q_2/\Z_2)^2,\\
 2&1&\Q_2/\Z_2,\\
 2&2&0.
\end{array}
\end{equation}
Consequently, the $p\equiv1\pmod8$ part of
\eqref{eq:conjecture} is equivalent to excluding the third row.
In particular, finiteness of $\Sha(E_p/\Q)[2^\infty]$ is a
sufficient additional hypothesis for $r(E_p)\in\{0,2\}$.
\end{proposition}

\begin{proof}
Here $s_2(E_p)=2$. In Lemma~\ref{lem:parity}, either
$(\rho_2,m)=(0,1)$ or $(2,0)$. In the first case $r=\delta=0$
and the finite quotient is one pair $(\Z/2^a\Z)^2$ with $a\ge1$.
In the second case the finite quotient vanishes and $r+\delta=2$,
which gives the other three rows. Only the third row has odd rank.
\end{proof}

\begin{remark}\label{rem:scope}
The alternatives in \eqref{eq:fourcases} are a classification of what
the argument permits, not a claim that all four occur. In particular,
no example of an infinite Tate--Shafarevich group is asserted. The
Cassels--Tate pairing on the finite quotient does not exclude a single
divisible summand $\Q_2/\Z_2$.
\end{remark}

\begin{corollary}\label{cor:evenlow}
Assume Hypothesis~\ref{hyp:bsd}. If $p\equiv1\pmod8$ and
$\rho_2(E_p)=0$, then $r_{\an}(E_p)=r(E_p)=0$,
$\Sha(E_p/\Q)$ is finite, and
\[
 L(E_p,1)=\frac{\Omega_p}{2}\,\#\Sha(E_p/\Q).
\]
Its $2$-primary part is $(\Z/2^a\Z)^2$ for some $a\ge1$.
\end{corollary}

\begin{proof}
Apply Hypothesis~\ref{hyp:bsd}, Proposition~\ref{prop:local}, and
the first row of \eqref{eq:fourcases}.
\end{proof}

Hypothesis~\ref{hyp:bsd} says nothing when $\rho_2(E_p)=2$.
In particular, it cannot be invoked merely from an a priori bound
$r(E_p)\le2$ or from the condition $r(E_p)=1$.
The unhandled case in \eqref{eq:conjecture} is therefore a genuine
logical boundary of this application.

\appendix
\section{A local descent in the original coordinates}\label{app:descent}

We give a direct proof of Proposition~\ref{prop:classical}. Its
purpose is to verify the finite-level input explicitly; the Selmer
dimensions themselves are the known results of \cite{QZ}. The
isogeny chosen here has kernel generated by $(0,0)$ on
\eqref{eq:curve}, rather than the kernel used on
\eqref{eq:translation} in \cite{QZ}.

Let $\phi:E_p\to E'_p$ be the degree-two isogeny with this kernel,
where
\[
 E'_p:\quad Y^2=X^3+2(p+2)X^2+q^2X,
 \qquad \hphi\phi=[2].
\]
Put $S_\phi=\Sel^\phi(E_p/\Q)$ and
$S_{\hphi}=\Sel^{\hphi}(E'_p/\Q)$; the curve in the notation is
always the \emph{source} of the isogeny. The usual isogeny descent
identifies these groups with locally soluble square classes for
\begin{align}
 C_d:\quad dW^2&=d^2U^4+2(p+2)dU^2V^2+q^2V^4,
 \label{eq:C}\\
 D_d:\quad dW^2&=d^2U^4-(p+2)dU^2V^2+2pV^4,
 \label{eq:D}
\end{align}
respectively \cite[Chapter X, Proposition 4.9]{Silverman}.
The equations are understood as genus-one double covers of the
$(U:V)$-line; a local point has $(U,V)\ne(0,0)$. Square classes are
represented by signed squarefree integers.

\subsection{Support and the dual Selmer group}

Local solubility forces
\begin{equation}\label{eq:support}
 S_\phi\subseteq\{1,-1,q,-q\},\qquad
 S_{\hphi}\subseteq\{1,2,p,2p\}.
\end{equation}
Here is a valuation proof of the support assertion. In an equation
\[
 dW^2=d^2U^4+bdU^2V^2+cV^4
\]
with integral $b,c$ and squarefree $d$, suppose a prime $\ell$ divides
$d$ but not $c$. Scale so that $U,V\in\Z_\ell$ and at least one is a
unit. The equation then forces $W\in\Z_\ell$. Reduction modulo
$\ell$ gives $\ell\mid V$, hence $U$ is a unit. The valuations of
the three terms on the right are respectively $2$, at least $3$,
and at least $4$. Thus the right side has valuation two, whereas
$\ord_\ell(dW^2)$ is odd, a contradiction. Applying this to
\eqref{eq:C} and \eqref{eq:D} proves that the squarefree
representatives divide $q^2$ and $2p$, respectively. If $d<0$ in
\eqref{eq:D}, the right side is strictly positive over $\R$ whenever
$(U,V)\ne(0,0)$, while the left side is nonpositive. This removes the
negative classes in the second group.

The dual Kummer map sends
\[
 O\mapsto1,\quad (0,0)\mapsto2p,\quad
 (2,0)\mapsto2,\quad(p,0)\mapsto p.
\]
Consequently,
\begin{equation}\label{eq:dualexact}
 S_{\hphi}=\{1,2,p,2p\},
\end{equation}
and every element is represented by a rational $2$-torsion point.

\subsection{Local obstructions in \texorpdfstring{$S_\phi$}{the phi-Selmer group}}

If $p\equiv3\pmod4$, the class $-1$ is excluded at $p$. Indeed,
\eqref{eq:C} for $d=-1$ can be written
\begin{equation}\label{eq:normobstruction}
 \bigl(U^2-(p+2)V^2\bigr)^2+W^2=8pV^4.
\end{equation}
Over $\Q_p$, a nonzero sum of two squares has even valuation because
$-1$ is not a square in the residue field. If $V\ne0$, the right
side has odd valuation; if $V=0$, the equation forces $U=W=0$.
Both are impossible.

If $\bigl(\frac2q\bigr)=-1$, both $q$ and $-q$ are excluded at
$q$. For $d=\eps q$, with $\eps\in\{1,-1\}$,
\eqref{eq:C} becomes
\begin{equation}\label{eq:qobstruction}
 W^2=\eps q(U^4+V^4)+2(p+2)U^2V^2.
\end{equation}
Normalize $U,V$ to be integral with at least one a unit. If both are
units, reduction modulo $q$ gives
$W^2\equiv8U^2V^2$, contrary to $\bigl(\frac2q\bigr)=-1$.
If just one is a unit, the right side has valuation one, again
impossible. This excludes both classes when $p\equiv5,7\pmod8$.

\subsection{The remaining locally soluble classes}

We first show that $-1\in S_\phi$ when $p\equiv1\pmod4$.
Writing its equation as
\[
 W^2=-U^4+2(p+2)U^2V^2-q^2V^4,
\]
one obtains the following local points. At $p$, take $V=0,U=1$,
since $-1$ is a square in $\Q_p$. At $2$, take $U=V=1$. The
right side is
\[
 8-(p-3)^2=4\left(2-\left(\frac{p-3}{2}\right)^2\right),
\]
and the expression in parentheses is $1\pmod8$ because
$(p-3)/2$ is odd. It is therefore a square in $\Q_2$.

At $q$ one has $q\equiv3\pmod4$. Surjectivity of the norm
$\F_{q^2}^{\times}\to\F_q^{\times}$ gives $q+1$ solutions
$(u,t)\in\F_q^2$ to $u^2+t^2=8$. At most four have a zero
coordinate, and $q\ge5$, so choose one with $ut\ne0$.
Then $(U,V,W)=(u,1,ut)$ satisfies the equation modulo $q$.
The derivative with respect to $W$ is nonzero, so Hensel's lemma
lifts the point to $\Q_q$. Over $\R$, take $V=1$ and
$U^2=p+2$, for which $W^2=8p>0$.

Next suppose $p\equiv1,3\pmod8$. For $d=q$ the equation is
\[
 W^2=q(U^4+V^4)+2(p+2)U^2V^2.
\]
At $p$, the choice $U=1,V=0$ works because
$\bigl(\frac qp\bigr)=\bigl(\frac{-2}p\bigr)=1$.
At $q$, take $U=V=1$, obtaining $W^2=4p$; this works because
$\bigl(\frac pq\bigr)=\bigl(\frac2q\bigr)=1$.
At $2$, if $p\equiv3\pmod8$, use $U=1,V=0$ and $q\equiv1\pmod8$.
If $p\equiv1\pmod8$, use $U=V=1$ and $W^2=4p$.
A real point exists since $q>0$.

For both quartics just considered, the discriminant of the polynomial
on the right is $2^{14}p^2q^2$. At any prime
$\ell\notin\{2,p,q\}$, the associated projective double cover
therefore has smooth genus-one reduction. The Hasse--Weil bound gives at least
$\ell+1-2\sqrt{\ell}>0$ rational points on the special fiber, and
smoothness lifts such a point to a $\Q_\ell$-point.
This completes the local checks at every place.

Combining membership, the obstructions, and the group law on square
classes gives
\begin{equation}\label{eq:phitable}
 S_\phi=
 \begin{cases}
 \{1,-1,q,-q\},&p\equiv1\pmod8,\\
 \{1,q\},&p\equiv3\pmod8,\\
 \{1,-1\},&p\equiv5\pmod8,\\
 \{1\},&p\equiv7\pmod8.
 \end{cases}
\end{equation}
For example, in the second line the membership of $q$ and the
exclusion of $-1$ also exclude $-q$.

\subsection{Recovering the full \texorpdfstring{$2$}{2}-Selmer group}

The cohomology sequence for
$0\to E_p[\phi]\to E_p[2]\to E'_p[\hphi]\to0$, with its
local conditions, gives an exact sequence through the middle terms
\begin{equation}\label{eq:isogenyselmer}
 0\longrightarrow
 \frac{E'_p(\Q)[\hphi]}{\phi(E_p(\Q)[2])}
 \longrightarrow S_\phi
 \longrightarrow\Sel_2(E_p/\Q)
 \longrightarrow S_{\hphi}.
\end{equation}
The quotient on the left vanishes: $E_p[2]$ is rational and its
image under $\phi$ is the order-two group $E'_p[\hphi]$.
The last map in \eqref{eq:isogenyselmer} is surjective in this
particular family. Indeed, by \eqref{eq:dualexact} each of its target
classes is the dual Kummer class of a rational $2$-torsion point,
and hence is the image of that point's ordinary $2$-Kummer class.
Thus
\[
 0\longrightarrow S_\phi\longrightarrow\Sel_2(E_p/\Q)
 \longrightarrow S_{\hphi}\longrightarrow0
\]
is short exact. Equations~\eqref{eq:dualexact} and
\eqref{eq:phitable} now give \eqref{eq:selmerdim}.
\qed

\section*{Acknowledgment of assistance}
This assistance is distinct from the external BSD input
\cite{OAI}, whose use is explicitly conditional throughout the paper.


\begin{thebibliography}{99}

\bibitem{Hatley}
Jeffrey Hatley,
\emph{On the rank of the elliptic curve $y^2=x(x-p)(x-2)$},
preprint (2009),
\href{https://arxiv.org/abs/0909.1614v1}{arXiv:0909.1614v1}.

\bibitem{Joshi}
Kirti Joshi,
\emph{On the analytic rank of the twin prime elliptic curve
$y^2=x(x-2)(x-p)$},
preprint (2025),
\href{https://arxiv.org/abs/2508.12340v1}{arXiv:2508.12340v1}.

\bibitem{OAI}
OpenAI,
\emph{Exact Birch--Swinnerton-Dyer formula from low Selmer corank},
OpenAI Math Release preprint, October 3, 2026.
\href{https://github.com/openai/math/blob/main/preprints/Exact-Birch-Swinnerton-Dyer-Formula-from-Low-Selmer-Corank-October-3-2026/exact-bsd-low-selmer-corank.pdf}{Manuscript in the OpenAI mathematics collection};
\href{https://github.com/openai/math/blob/main/CONTENTS.md}{announced statement in result family 002}.
Release consulted October 7, 2026.
The stated low-corank implication is used here as
Hypothesis~\ref{hyp:bsd}.

\bibitem{PS}
Bjorn Poonen and Michael Stoll,
\emph{The Cassels--Tate pairing on polarized abelian varieties},
Annals of Mathematics (2) \textbf{150} (1999), no.~3, 1109--1149;
\href{https://arxiv.org/abs/math/9911267}{arXiv:math/9911267}.

\bibitem{QZ}
Derong Qiu and Xianke Zhang,
\emph{Mordell--Weil groups and Selmer groups of twin-prime elliptic curves},
Science in China Series A: Mathematics \textbf{45} (2002), no.~11,
1372--1380,
\href{https://doi.org/10.1007/BF02880031}{doi:10.1007/BF02880031}.
Preprint version entitled \emph{Mordell--Weil groups and Selmer groups
of two types of elliptic curves},
\href{https://arxiv.org/abs/math/0103243v1}{arXiv:math/0103243v1} (2001).
The theorem numbers cited in this note refer to that preprint version.

\bibitem{Silverman}
Joseph H. Silverman,
\emph{The arithmetic of elliptic curves},
Graduate Texts in Mathematics, vol.~106,
Springer-Verlag, New York, 1986.

\end{thebibliography}
\end{document}